\documentclass [12pt,a4paper,reqno]{amsart}

\usepackage[usenames]{color}
\newcommand\Cref[1]{{Corollary~\ref{#1}}}
\newcommand{\etype}[1]{\renewcommand{\labelenumi}{(#1{enumi})}}
\def\eroman{\etype{\roman}}

\newtheorem{thm}{Theorem}[section] 

\newtheorem{cor}[thm]{Corollary}

\newtheorem{prop}[thm]{Proposition}

\newtheorem{rem}[thm]{Remark}

\def\ann{\operatorname {Ann}}

\def\Ann{\operatorname {Ann}}

\def\Nil{\operatorname {Nil}}
\def\({\left(}
\def\){\right)}

\def\End{\operatorname{End}}
\def\cent{\operatorname{Cent}}

\long\def\forget#1\forgotten{}

\newcommand\comp[3][\bullet]{{{#1}_{{\if1#2{}\else{#2}\fi}{\if#3K{}\else{(#3)}\fi}}}} 

\begin{document}

\title{Representability of PI-algebras}

\author{Be'eri Greenfeld}
\address{Department of Mathematics, Bar-Ilan University, Ramat-Gan
52900, Israel} %
\email{beeri.greenfeld@gmail.com}

\author{Louis Rowen}
\address{Department of Mathematics, Bar-Ilan University, Ramat-Gan
52900, Israel} %
\email{rowen@math.biu.ac.il}
 \thanks{The authors thank Lance Small for many helpful insights}. \thanks{This work was supported by
the U.S.-Israel Binational Science Foundation (grant no. 2010149)}.
\begin{abstract}
    This note concerns the still open question of representability of Noetherian PI-algebras.
    Extending a result of Rowen and Small (with an observation of Bergman)
    that every finitely generated module over a commutative Noetherian ring containing a field is representable,  we provide a representability machinery for a Noetherian PI-algebra $R$ containing a field, which includes the case that $R$ is finite (as a module) over a commutative
 subalgebra isomorphic to $R/N$.
    We  construct a family of non-representable PI-algebras demonstrating the sharpness of these results, as well as of some well known previous representability results.
\end{abstract}

\maketitle

\section{Introduction}

One of the major problems in the theory of algebras satisfying
polynomial identities (PI-algebras) is to determine whether a given
PI-algebra is representable, namely, whether it embeds into a matrix
ring over a field, or more generally a commutative ring. To
distinguish between the two cases, we call the latter case weakly
representable (after \cite{RowenSmall}).



We produce a representability machinery, allowing one to deduce
that, under certain conditions, a given Noetherian PI-algebra is
representable, including the case that $R$ is finite over a
commutative subalgebra isomorphic to $R/N$. The significance of this
result is in  that there is an affine PI-algebra satisfying ACC on
ideals satisfying this property, which is not weakly representable
(\S\ref{affine_example}).

In another direction, a ring $R$ with nilradical
$\Nil(R)\triangleleft R$ is \textit{semiprimary} if $R/\Nil(R)$ is
semisimple Artinian. Amitsur-Rowen-Small proved that a semiprimary
PI-algebra whose radical squared zero is weakly representable
(appears in \cite[Section~6]{RowenSmall}). We present a semiprimary
PI-algebra with radical cubed zero, which is non-weakly
representable (\S\ref{semip}).

Even the question of whether arbitrary Artinian PI-algebras are
representable still seems to be open \cite[Question 5 in page
388]{book}.

\section{Representability}

\subsection{A representability machinery}

We write $N$ for $\Nil(R)$, which is nilpotent when $R$ is left
Noetherian. We present this result in a fair amount of generality.

\begin{thm} \label{machinery2}
Let $R$ be a left Noetherian algebra over a field, which contains a
weakly Noetherian subalgebra $W\subseteq R$ such that $R/N$ is a
finitely generated left module over~ $\overline{W}$, the reduction
of $W$ modulo $N$, satisfying the condition:

  $W[c^{-1}]$ is finite over its center, for some
 $c$ of $C: = \cent (W)$ which is regular in $R$.

  Then $R$ is representable.
\end{thm}
\begin{proof} 
Let $d$ be the nilpotency index of $N$. Consider $R$ as a (left)
$W$-module via the natural action of~$W$, being a subalgebra of $R$.

Since $R/N$ is finite over $\overline{W}$, we can write:
$$R/N=\overline{W}\overline{r}_1+\cdots\overline{W}\overline{r}_q$$
Pick arbitrary lifts $r_1,\dots,r_q\in R$ such that
$\overline{r}_i=R_i+N$. Therefore:
$$R=Wr_1+\cdots+Wr_q+N.$$
We now prove by induction on $e$ that $N^{d-e}$ is a finitely
generated $W$-module. First take $e=1$. Since $R$ is left
Noetherian, we can write $N^{d-1}$ as a finitely generated
$R$-module:
$$N^{d-1}=\sum_{i=1}^{k} Ru_i =\sum_{i=1}^{k}
(Wr_1+\cdots+Wr_q+N)u_i=\sum_{i=1}^{k} (Wr_1+\cdots+Wr_q)u_i
$$
where the third equality follows since $Nu_1,\dots,Nu_k\subseteq
N^d=0$.

Now assume that the claim was proved for $e'<e$. Again, since $R$ is
left Noetherian we can write $N^{d-e}$ as a finitely generated
$R$-module:
$$N^{d-e}=\sum_{i=1}^{l} Rv_i =\sum_{i=1}^{l} (Wr_1+\cdots+Wr_q+N)v_i=\sum_{i=1}^{l} (Wr_1+\cdots+Wr_q)v_i+N^{d-e+1}$$
But by the induction hypothesis, $N^{d-e+1}$ is a finitely generated
$W$-module, so $N^{d-e}$ is finitely generated as well. In
particular, we get that $N$ itself is a finitely generated
$W$-module, and therefore so is $R$.

 We can localize $R[c^{-1}]$ to be a finite
module over $W[c^{-1}]$, and thus over   $C[c^{-1}]$, which is
Noetherian by the Eakin-Formanek theorem \cite[Theorem~5.1.12]{Row}
since $W[c^{-1}]$ is weakly Noetherian. But then
$\End_{C[c^{-1}]}R[c^{-1}]$ is representable by
\cite[Corollary~3.4]{RowenSmall}. Now $R^{op}$ acts by right
multiplication on itself, and this action is $C$-equivariant.
Therefore $R^{op}$ is representable; considering the transpose of
this representation, we obtain that $R$ is representable as well.
\end{proof}

Note that the condition satisfied in either of the following
situations:

\begin{enumerate}\eroman
\item $W$ already is finite over its center, in particular if $W$ is commutative;
\item $W$ is a semiprime PI-algebra and any   regular central element
$c$  of $W$ is regular in $R$, since there is such $c$ as seen via
\cite[Theorem~1.8.48]{Row}. Note that if $R$ is uniform as a
$W,C$-bimodule then
 by Fitting's  lemma, $\Ann _R (c^k)=0$ for some $k$
(since $Rc^k \ne 0$), so $\Ann _R (c)=0$.
\end{enumerate}

\begin{cor}  Suppose $R$ is    a left Noetherian algebra over a field, and
 $R/N$ is finite over a commutative polynomial ring $F[x].$ Then $R$ is representable.\label{linear_growth}
\end{cor}
\begin{proof} Lift $x$ to an element $c$ of $R$. Then $W := F[c]$  is
a commutative algebra  satisfying the hypothesis of the theorem,
since $\bar W \approx F[x]$ is Noetherian.
\end{proof}

\begin{cor} \label{linear_growth2}
Let $R$ be a left Noetherian algebra with nilpotent radical $N$,
such that $R/N$ is affine of linear growth. Then $R$ is
representable.
\end{cor}

\begin{proof}
Since $R/N$ is affine of linear growth, by \cite{SSW} it is a
finitely generated module over a central polynomial ring in one
variable $F[t]\subseteq R/N$.
\end{proof}

\section{Examples}

\subsection{A non-representable affine PI-algebra with ACC on ideals} \label{affine_example}

We give an example of a non-weakly representable affine PI-algebra
satisfying ACC on ideals. Moreover, the quotient of our example by
its nilpotent radical is a polynomial ring in one variable, thus
emphasizing the sharpness of Corollary \ref{linear_growth}.
(Compare with \cite{Markov}.) Let $A$ be an $F$-algebra and $M$ an
$A$-bimodule. Given an $F$-linear map $B:M\otimes_A M\rightarrow F$,
we can define an $F$-algebra:

\begin{equation*}
R = \left(\begin{matrix}
        F & M & F\\
        0 & A & M \\
        0 & 0 & F
\end{matrix}\right)
\end{equation*}
whose multiplication is given by:
\begin{equation*}
\left(\begin{matrix}
        \alpha_1 & v & \lambda \\
        0 & f & w \\
        0 & 0 & \alpha_2
\end{matrix}\right)
\left(\begin{matrix}
        \alpha_1' & v' & \lambda' \\
        0 & f' & w' \\
        0 & 0 & \alpha_2'
\end{matrix}\right) =
\left(\begin{matrix}
        \alpha_1\alpha_1' & \alpha_1v'+vf' & \alpha_1\lambda'+\alpha_2'\lambda+B(v,w') \\
        0 & ff' & fw'+\alpha_2'w \\
        0 & 0 & \alpha_2\alpha_2'
\end{matrix}\right)
\end{equation*}

Since $B$ is $A$-equivariant, namely $B$ is defined over $M\otimes_A
M$, this multiplication law endows $R$ with a well-defined
 $F$-algebra structure. To check associativity, note that the only
 difficulty would be checking the products involving the  $1,2$ and $2,3$
 positions, but
 $$(\alpha_1 e_{1,1}\cdot v' e_{1,2})\cdot w''e_{2,3} = B(\alpha_1 v', w'') =
 \alpha_1 B(v', w'') = (\alpha_1 e_{1,1}\cdot (v' e_{1,2}\cdot w''e_{2,3}).$$

We now specify $A$ and $M$. Let $A=F[t]$ and let
$M=Fu_1+Fu_2+\cdots$ be a countable dimensional $F$-vector space. We
consider $M$ as an $F[t]$-bimodule through:
$$tu_i=ut_i=u_{i+1}$$

We now specify $B$, writing it as a bilinear form $B$. Set:
$$B(u_i,u_j)=\begin{array}{cc}
  \Bigg\{ &
    \begin{array}{cc}
      1, & \text{if}\ \ \ \exists t\geq 1:\ i+j=2^t \\
      0, & \text{otherwise}
    \end{array}
\end{array}$$
It is easy to verify that:
$$B(u_it^j,u_k)=B(u_{i+j},u_k)=B(u_i,u_{j+k})=B(u_i,t^ju_k)$$
So $B$ is a well-defined $F$-linear map defined over $M\otimes_F[t]
M$ and thus $R$ is a well defined $F$-algebra:
\begin{equation*}
R = \left(\begin{matrix}
        F & M & F\\
        0 & F[t] & M \\
        0 & 0 & F
\end{matrix}\right)
\end{equation*}

\begin{prop}
The algebra $R$ is an affine PI-algebra whose radical cubed zero and
$R/N\cong F\times F[t]\times F$.
\end{prop}
\begin{proof}
The algebra $R$ is generated over $F$ by the following elements:
$$
E_{1,1}=\left(\begin{matrix}
        1 & 0 & 0\\
        0 & 0 & 0 \\
        0 & 0 & 0
\end{matrix}\right),\ E_{1,3}=\left(\begin{matrix}
        0 & 0 & 1\\
        0 & 0 & 0 \\
        0 & 0 & 0
\end{matrix}\right),\
E_{2,2}=\left(\begin{matrix}
        0 & 0 & 0\\
        0 & 1 & 0 \\
        0 & 0 & 0
\end{matrix}\right),\
E_{3,3}=\left(\begin{matrix}
        0 & 0 & 0\\
        0 & 0 & 0 \\
        0 & 0 & 1
\end{matrix}\right)
$$
$$
\left(\begin{matrix}
        0 & 0 & 0\\
        0 & t & 0 \\
        0 & 0 & 0
\end{matrix}\right),\ \ \
\left(\begin{matrix}
        0 & u_1 & 0\\
        0 & 0 & 0 \\
        0 & 0 & 0
\end{matrix}\right),\ \ \
\left(\begin{matrix}
        0 & 0 & 0\\
        0 & 0 & u_1 \\
        0 & 0 & 0
\end{matrix}\right)
$$
Observe that the nilpotent radical of $R$ is:
\begin{equation*}
N=\left(\begin{matrix}
        0 & M & F\\
        0 & 0 & M \\
        0 & 0 & 0
\end{matrix}\right)
\end{equation*}
So $N^3=0$ and $R/N\cong F\times F[t]\times F$. Moreover, $R$
satisfies the polynomial identity:
$$[X_1,X_2][X_3,X_4][X_5,X_6]=0,$$
as claimed.
\end{proof}

\begin{prop}
The algebra $R$ does not satisfy ACC on (left) annihilators.
\end{prop}
\begin{proof}
For a vector $w\in M$ denote: $w^{\perp}=\{v\in M|\ B(v,w)=0\}$.
Notice that:
\begin{equation*}
\ann \left(\begin{matrix}
        0 & 0 & 0 \\
        0 & 0 & u_i \\
        0 & 0 & 0
\end{matrix}\right) = \left(\begin{matrix}
        F & u_i^{\perp} & F \\
        0 & 0 & M \\
        0 & 0 & F
\end{matrix}\right)
\end{equation*}
Consider the sets:
\begin{equation*}
S_n = \Bigg\{\left(\begin{matrix}
        0 & 0 & 0 \\
        0 & 0 & u_{2^n}\\
        0 & 0 & 0
\end{matrix}\right), \left(\begin{matrix}
        0 & 0 & 0 \\
        0 & 0 & u_{2^{n+1}} \\
        0 & 0 & 0
\end{matrix}\right),\dots\Bigg\}
\end{equation*}
It is easy to see that: $$\left(\begin{matrix}
        0 & u_{2^n} & 0 \\
        0 & 0 & 0 \\
        0 & 0 & 0
\end{matrix}\right)\in \ann(S_{n+1})\setminus \ann(S_n)$$
since $2^n+2^m$ is not a power of $2$ for $m>n$. Therefore we have
an infinite strictly increasing ascending chain of annihilators:
$$\ann(S_1)\subset \ann(S_2)\subset \cdots$$
So $R$ does not satisfy ACC on left annihilators.
\end{proof}
We now consider two-sided ideals of $R$.

\begin{prop}
The algebra $R$ satisfies ACC on two sided ideals.
\end{prop}
\begin{proof}
Let $I\triangleleft R$, fix some $r\in I$, and write:
$$r = \left(\begin{matrix}
        \alpha_1 & v & \lambda \\
        0 & f(t) & w\\
        0 & 0 & \alpha_2
\end{matrix}\right)$$

Then:
$$E_{1,1}rE_{1,1}=\left(\begin{matrix}
        \alpha & 0 & 0 \\
        0 & 0 & 0\\
        0 & 0 & 0
\end{matrix}\right),\ \ \ E_{2,2}rE_{2,2}=\left(\begin{matrix}
        0 & 0 & 0 \\
        0 & f(t) & 0\\
        0 & 0 & 0
\end{matrix}\right)$$ $$E_{3,3}rE_{3,3}=\left(\begin{matrix}
        0 & 0 & 0 \\
        0 & 0 & 0\\
        0 & 0 & \beta
\end{matrix}\right),\ \ \ E_{1,1}rE_{3,3}=\left(\begin{matrix}
        0 & 0 & \lambda \\
        0 & 0 & 0\\
        0 & 0 & 0
\end{matrix}\right)$$
$$\left(E_{2,2}-1\right)rE_{2,2}=\left(\begin{matrix}
        0 & v & 0 \\
        0 & 0 & 0\\
        0 & 0 & 0
\end{matrix}\right),\ \ \ E_{2,2}r\left(E_{2,2}-1\right)=\left(\begin{matrix}
        0 & 0 & 0 \\
        0 & 0 & w\\
        0 & 0 & 0
\end{matrix}\right)$$
Moreover:
$$
\left(\begin{matrix}
        0 & v & 0 \\
        0 & 0 & 0\\
        0 & 0 & 0
\end{matrix}\right)
\left(\begin{matrix}
        0 & 0 & 0 \\
        0 & p(t) & 0\\
        0 & 0 & 0
\end{matrix}\right)=\left(\begin{matrix}
        0 & vp(t) & 0 \\
        0 & 0 & 0\\
        0 & 0 & 0
\end{matrix}\right),\ \ \
\left(\begin{matrix}
        0 & 0 & 0 \\
        0 & p(t) & 0\\
        0 & 0 & 0
\end{matrix}\right)
\left(\begin{matrix}
        0 & 0 & 0 \\
        0 & 0 & w\\
        0 & 0 & 0
\end{matrix}\right)=\left(\begin{matrix}
        0 & 0 & 0 \\
        0 & 0 & p(t)w\\
        0 & 0 & 0
\end{matrix}\right)
$$
Then:
$$r\cdot \left(\begin{matrix}
        0 & 0 & 0 \\
        0 & t^i & 0\\
        0 & 0 & 0
\end{matrix}\right)=\left(\begin{matrix}
        0 & vt^i & 0 \\
        0 & f(t)t^i & 0\\
        0 & 0 & 0
\end{matrix}\right)$$
And:
$$\left(\begin{matrix}
        0 & 0 & 0 \\
        0 & p(t) & 0\\
        0 & 0 & 0
\end{matrix}\right)\cdot r\cdot \left(\begin{matrix}
        0 & 0 & 0 \\
        0 & 1 & 0\\
        0 & 0 & 0
\end{matrix}\right)=\left(\begin{matrix}
        0 & 0 & 0 \\
        0 & p(t) & 0\\
        0 & 0 & 0
\end{matrix}\right)\cdot \left(\begin{matrix}
        0 & vq(t) & 0 \\
        0 & f(t)q(t) & 0\\
        0 & 0 & 0
\end{matrix}\right)=
\left(\begin{matrix}
        0 & 0 & 0 \\
        0 & f(t)p(t) & 0\\
        0 & 0 & 0
\end{matrix}\right)$$
It follows that $I$ contains a subspace of the form:
\begin{equation*}
\left(\begin{matrix}
        E & V & L \\
        0 & J & W \\
        0 & 0 & K
\end{matrix}\right)
\end{equation*}
where $E,K,L\in \{0,F\}$, $J\triangleleft F[t]$ and $V,W$ are
$F[t]$-sub-bimodules of $M$. Moreover, $V,W$ are non-zero if and
only if $v,w$ are non-zero, respectively. Notice that any non-zero
ideal of $F[t]$ is finite codimensional, and any non-zero
sub-bimodule of $M$ is finite codimensional. Therefore we have that
$I$ is finite codimensional inside a subspace of the form:
\begin{equation*}
\left(\begin{matrix}
        E & V & L \\
        0 & J & W \\
        0 & 0 & K
\end{matrix}\right)
\end{equation*}
where $E,K,L\in \{0,F\}$, $J\in \{0,F[t]\}$ and $V,W\in \{0,M\}$. It
follows that any ascending chain of ideals of $R$ stabilizes.
\end{proof}

\begin{rem}
We can modify the algebra $R$ constructed above to be moreover
irreducible with $R/N\cong F[t]$, since an algebra with ACC on
ideals is a subdirect product of finitely many irreducible algebras,
and a subdirect product of irredcible (weakly) representable
algebras is again (weakly) representable.
\end{rem}

\subsection{A non weakly-representable semiprimary PI-algebra}\label{semip}

We now extend the construction from Subsection \ref{affine_example}.
We take $A=F(t)$ and $M=V$ a $1$-dimensional $F(t)$-vector space,
which we naturally identify with $F(t)$. Then $V\otimes_{F(t)}
V\cong V$ as an $F$-vector space by $v\otimes w\mapsto vw$. We fix
an $F$-linear basis for $F(t)$, say, $\mathfrak{B}$ containing
$1,t,t^2,\dots$ and define $\widetilde{B}:V\otimes_{F(t)}
V\rightarrow F$ on basis elements as follows:

$$\widetilde{B}(1,v)=\begin{array}{cc}
  \Bigg\{ &
    \begin{array}{cc}
      1, & \text{if}\ \ \ \exists k\geq 1:\ v=t^{2^k} \\
      0, & \text{otherwise}
    \end{array}
\end{array}$$

We can therefore form, in the same manner as of Subsection
\ref{affine_example}, the following $F$-algebra:

\begin{equation*}
S = \left(\begin{matrix}
        F & V & F\\
        0 & F(t) & V \\
        0 & 0 & F
\end{matrix}\right)
\end{equation*}

Notice that
%
the $F$-algebra
$R$ constructed in Subsection \ref{affine_example} embeds into $S$:

\begin{equation*}
R = \left(\begin{matrix}
        F & M & F\\
        0 & F[t] & M \\
        0 & 0 & F
\end{matrix}\right)\hookrightarrow
\left(\begin{matrix}
        F & V & F\\
        0 & F(t) & V \\
        0 & 0 & F
\end{matrix}\right) = S
\end{equation*}
Notice that, if $N\triangleleft S$ is the nilpotent radical of $S$,
then $S/N\cong F\times F(t)\times F$, so $S$ is a semiprimary
PI-algebra. Since $R$ is non-weakly representable, we get that so is
$S$.

\end{document}